\documentclass[a4paper,12pt]{amsart}
\usepackage[english]{babel}
\usepackage{amsmath,amsthm,amssymb,amsfonts}
\usepackage{multicol}
\usepackage[titletoc]{appendix}
\usepackage{soul}

\usepackage[pdftex]{color}
\usepackage[bookmarks=true,hyperindex,pdftex,colorlinks, citecolor=blue,linkcolor=blue, urlcolor=blue]{hyperref}
\usepackage[dvipsnames]{xcolor}
\usepackage{mathtools}
\usepackage{graphicx}
\usepackage{enumitem}
\usepackage{pgf}
\usepackage{tikz-cd}
\usepackage[normalem]{ulem}

\usepackage{tikz}
\usetikzlibrary{mindmap,backgrounds, calc}
\usetikzlibrary{matrix,chains,scopes,positioning,arrows,fit}
\usetikzlibrary{positioning, shapes, patterns}
\usetikzlibrary{automata} 
\usetikzlibrary{shapes.geometric, arrows, arrows.meta}
\usetikzlibrary{positioning,decorations.markings}

\usetikzlibrary{matrix}\parskip=1ex
\newtheorem{theorem}{Theorem}[section]
\newtheorem{lemma}[theorem]{Lemma}

\newtheorem{proposition}[theorem]{Proposition}
\newtheorem{cor}[theorem]{Corollary}

\theoremstyle{definition}

\numberwithin{equation}{section}

\newcommand{\N}{\mathbb{N}}

\newcommand{\Ho}{\mathcal{H}}

\newcommand{\M}{\mathfrak{M}}
\newcommand{\A}{\mathcal{A}}
\newcommand{\Pol}{\mathcal{P}}

\DeclareMathOperator{\re}{Re}
\DeclareMathOperator{\Ext}{Ext}

\newcommand{\nn}[1]{{\left\vert\kern-0.25ex\left\vert\kern-0.25ex\left\vert #1 
		\right\vert\kern-0.25ex\right\vert\kern-0.25ex\right\vert}}

\renewcommand{\geq}{\geqslant}
\renewcommand{\leq}{\leqslant}

\newcommand{\eps}{\varepsilon}

\begin{document}
%\setcounter{tocdepth}{1}
%\date{\today}

%\renewcommand{\baselinestretch}{1.1}

\title[Boundaries for Gelfand transform images of Banach algebras]{Boundaries for Gelfand transform images of Banach algebras of holomorphic functions}

\author[Choi]{Yun Sung Choi}
\address[Y. S. Choi]{Department of Mathematics, POSTECH, Pohang 790-784, Republic of Korea \newline
	\href{http://orcid.org/0000-0000-0000-0000}{}}
\email{\texttt{mathchoi@postech.ac.kr}}

	\author[Jung]{Mingu Jung}
	\address[Jung]{Basic Science Research Institute and Department of Mathematics, POSTECH, Pohang 790-784, Republic of Korea \newline
		\href{http://orcid.org/0000-0000-0000-0000}{ORCID: \texttt{0000-0003-2240-2855} }}
	\email{\texttt{jmingoo@postech.ac.kr}}

\thanks{The first author was supported by NRF (NRF-2019R1A2C1003857) and the second author was supported by NRF (NRF-2019R1A2C1003857) and by POSTECH Basic Science Research Institute Grant (NRF-2021R1A6A1A10042944)}


\date{\today}


\begin{abstract}
Let $\A$ be a Banach algebra of bounded holomorphic functions on the open unit ball $B_X$ of a complex Banach space $X$. Considering the Gelfand transform image $\widehat{\A}$ of the Banach algebra $\A$, which is a uniform algebra on the spectrum of $\A$, we obtain an explicit description of the Shilov boundary for $\widehat{\A}$ for classical Banach spaces $X$ in the case where $\A$ is a certain Banach algebra, for instance, $\A_\infty (B_X)$, $\A_u (B_X)$ or $\A_{wu} (B_X)$. Some possible application of our result to the famous Corona theorem is also briefly discussed. 
\end{abstract}
\maketitle

%Nuevo

\section{Introduction}

For a complex Banach space $X$, let $\mathcal{A}$ be a Banach algebra of bounded holomorphic functions on the open unit ball $B_X$ endowed with the usual supremum norm. In contrast to a classical result due to G.E.~Shilov \cite{Shilov}, the Shilov boundary for $\A$ may not exist since the closed unit ball $\overline{B}_X$ is not compact unless $X$ is finite dimensional (all the definitions will be given later in Section \ref{sec:pre}). For instance, J.~Globevnik observed in \cite{Glob2} that there is no Shilov boundary for $\A_u (B_{c_0})$ and $\A_\infty (B_{c_0})$ by showing that there exists a sequence of closed boundaries whose intersection is empty. Since then, many authors have studied the Shilov boundary and obtained various characterizations of boundaries for $\A_u (B_X)$ and $\A_\infty (B_X)$ in the case, for instance, when $X$ is $\ell_p$ with $1 \leq p \leq \infty$, $C(K)$, the predual of the Lorentz sequence space, or a Marcinkiewicz sequence space \cite{Acosta, AcoLou, AcoMor, AMR, ACLP, CGKM, CH, CHL, MR1, MR2}. 

Our intention here is to study the Shilov boundary for the Gelfand transform image $\widehat{\A}$ of $\A = \A_u (B_X)$ or $\A_\infty (B_X)$ which turns to be a uniform algebra on the spectrum (maximal ideal space) $\M(\A)$ of the Banach algebra $\A$. It is well known that if $A$ is an algebra of bounded continuous functions on a Hausdorff space $\Omega$, then $A$ is isometric to $\widehat{A} := \{ \widehat{f} : f \in A\}$, where $\widehat{f} (\phi ) = \phi (f)$ for every $\phi \in \M (A)$, which is a uniform algebra on its spectrum $\M(A)$ via the Gelfand transform (see, for instance, \cite{Lee}). 
As we consider a uniform algebra $\widehat{\A}$, there is always a minimal closed boundary $\Gamma \subset \M (\A)$ with the property that every elements $\widehat{f}$ of $\widehat{\A}$ attains its maximum absolute value at some point of $\Gamma$ due to the aforementioned result of G.E.~Shilov \cite{Shilov}. 
From this point of view, it is natural to investigate how closed boundaries for $\A$ and $\widehat{\A}$ relate to each other, and the relation between the Shilov boundaries for $\A$ and $\widehat{\A}$ when the Shilov boundary for $\A$ exists. %We  
%a can be viewed as an algebra of continuous functions on $B_X$ and $\A$ is isometric to a uniform algebra on its spectrum (maximal ideal space) $\M (\A)$  That is, $\A$ is isometric to $\widehat{\A} := \{ \widehat{f} : f \in \A \}$,  which is a closed separating subalgebra of $C_b (\M(\A))$. In this case,   

The paper has the following structure. In Section \ref{section_1}, beginning by noticing some relation between those boundaries for $\A$ and $\widehat{\A}$, we obtain an explicit description of the Shilov boundary for $\widehat{\A}$ when $\A = \A_u (B_X), \A_\infty (B_X)$ or $\A_{wu} (B_X)$ for many classical Banach spaces $X$. For example, unlike the result of Globevnik on $\A_u (B_{c_0})$, we observe that the Shilov boundary for $\widehat{\A_u (B_{c_0})}$ coincides with $\{ \delta_z : z \in \mathbb{T}^\mathbb{N}\}$, where $\mathbb{T}^\mathbb{N} = \{ z = (z_n) \in \overline{B}_{\ell_\infty} : |z_n| = 1 \text{ for all } n \in \N \}$. Beside the case of $X= c_0$, the Shilov boundary for the Gelfand transform image of $\A_u (B_X)$ or $\A_\infty (B_X)$ when $X$ is locally uniformly convex Banach spaces, $\ell_1$, $d(w,1)$ (the Lorentz sequence space), locally $c$-uniformly convex order continuous sequence space, $C(K)$ with $K$ infinite countable compact Hausdorff space or $C_1 (H)$ (the space of trace class operators) is concretely obtained. 
At the end of Section \ref{section_1}, we discuss a possible application of our results on the Shilov boundary for $\widehat{\A_u (B_X)}$ to the famous Corona theorem in the context of infinite dimensional Banach spaces. In Section \ref{sec:wu}, we describe the Shilov boundary for the Gelfand transform image of $\A_{wu} (B_X)$ for a large class of Banach spaces $X$, for example, Banach spaces $X$ whose dual spaces $X^*$ are separable, or Banach spaces $X$ which have the Radon-Nikod\'ym property and their dual spaces $X^*$ have the approximation property.

\section{Preliminaries}\label{sec:pre}

Throughout the paper, all Banach spaces are assumed to be complex Banach spaces. Given a Banach space $X$, the unit sphere and the open unit ball are denoted by $S_X$ and $B_X$, respectively.
We are going to deal with the following Banach algebras of holomorphic functions: 
\begin{enumerate}
\setlength\itemsep{0.4em}
%\item[(a)] $\Ho^\infty (B_X) = \{ f : B_X \rightarrow \mathbb{C} : f \text{ is bounded and holomorphic on } B_X\}$,
\item[(a)] $\A_\infty (B_X) = \{ f \in \Ho^\infty (B_X) : f \text{ is continuously extended to } \overline{B}_X\}$,
\item[(b)] $\A_u (B_X) = \{ f \in \Ho^\infty (B_X) : f \text{ is uniformly continuous on } B_X \}$,
\item[(c)] $\A_{wu} (B_X) = \{ f \in \Ho^\infty (B_X) : f \text{ is weakly uniformly continuous on } B_X \}$, 
\end{enumerate} 
where each of them is endowed with the supremum norm on $B_X$. 
Notice the following inclusions: 
\[
\A_{wu} (B_X) \subseteq \A_u (B_X) \subseteq \A_\infty (B_X). %\subseteq \Ho^\infty (B_X). 
\]
It is clear that $\A_{wu} (B_X) = \A_u (B_X) = \A_\infty (B_X)$ when $X$ is a finite dimensional Banach space. On the other hand, it is known that those inclusions are strict in general when $X$ is infinite dimensional (see \cite[Section 12]{ACG91}). In addition, let us denote by $\A_a (B_X)$ the subalgebra of $\A_u (B_X)$ which consists of all approximable holomorphic functions on $B_X$. That is, $\A_a (B_X)$ is the closure in $\A_u (B_X)$ of the subalgebra generated by the constant functions and $X^*$. 

Given a Banach algebra $\A$, let us denote by $\M (\A)$ the spectrum of $\A$. That is, $\M(\A)$ is the space of all homomorphisms on $\A$ endowed with the Gelfand topology $\sigma$ (i.e., restriction of the weak-star topology). 
%It is well known that $\M (\A)$ is always a compact space. 
If we restrict ourselves to the case when $\A$ is one of the above three Banach algebras $\A_\infty (B_X)$, $\A_u (B_X)$ and $\A_{wu} (B_X)$, there is a canonical mapping $\pi : \M (\A) \rightarrow \overline{B}_{X^{**}}$ which is given by $\pi (\phi) = \phi \vert_{X^{*}}$ for every $\phi \in \M (\A)$. Moreover, if $z \in \overline{B}_{X^{**}}$, we can consider the point evaluation homomorphism $\delta_z$ in $\M(\A)$, that is, $\delta_z (f) =  \widetilde{f} (z)$ for every $f \in \A$, where $ \widetilde{f}$ is the Aron-Berner extension of $f$ \cite{AB}. %Let us mention that if, in addition, we assume that $\A \subseteq \A_\infty (B_X)$, then $\delta_z$ can be defined for all $z \in \overline{B}_{X^{**}}$.       

%and a bounded holomorphic function $f$ on the open unit ball $B_X$, let us denote by $ \widetilde{f}$ the Aron-Berner extension of $f$, i.e., $ \widetilde{f}$ is a bounded holomorphic function on $B_{X^{**}}$ which extends the function $f$.

Let $A \subset C_b (\Omega)$ be an algebra of functions on a topological space $\Omega$, where $C_b (\Omega)$ denotes the space of continuous and bounded functions on $\Omega$ endowed with the usual supremum norm on $\Omega$. Recall that a point $x \in \Omega$ is a \emph{strong boundary point for $A$} if, for each open neighborhood $U$ of $x$, there exists $f \in A$ with $f(x)=\sup_\Omega |f| = 1$ and $\sup_{y \in \Omega \setminus U} |f(y)| < 1$. 
A point $x \in \Omega$ is a \emph{peak point} for $A$ if there exists $f \in A$ such that $f(x)=1$ and $|f(y)| < 1$ for every $y \in \Omega \setminus \{x\}$. 
Moreover, a point $x \in \Omega$ is a \emph{strong peak point} for $A$ if there exists $f \in A$ such that $f(x)=1$ and for each open neighborhood $U$ of $x$, $\sup_{y \in \Omega \setminus U} |f(y)| < 1$. 
We say that a subset $\Gamma$ of $\Omega$ is a \emph{boundary} for the function algebra $A$ (in the Globevnik sense \cite{Glob1}) if 
\[
\| f \| = \sup\{ |f(x)| : x \in \Gamma \}, \quad \text{for all } f \in A. 
\]
If the intersection of all the closed boundaries for $A$ is again a boundary for $A$, then it is called the \emph{Shilov boundary} for $A$, and denoted by $\partial A$. A classical result by Shilov states that if $\Omega$ is compact and $A \subset C_b (\Omega)$ is a separating subalgebra (i.e., a uniform algebra), then the intersection of all closed boundaries is again a closed boundary, which becomes a minimal closed boundary for $A$ \cite{Shilov}. 
Note that when $\Omega$ is a compact topological space, each peak point for $A$ is a strong boundary point for $A$ and any strong boundary point for $A$ belongs to each closed boundary for $A$. 
For background information on function algebras, we refer to \cite{Dales, Stout}.

In this paper, for a Banach space $X$ and a Banach algebra $\mathcal{A}$ with $\A_a (B_X) \subseteq \A \subseteq \A_\infty (B_X)$, we shall use the following notations: 
\begin{align*}
\rho \mathcal{A} &= \{ z \in \overline{B}_X : z \text{ is a peak point for } \A \}, \\
\rho \widehat{\A} &= \{ \phi \in \M (\A) : \phi \text{ is a peak point for } \widehat{\A} \}, 
\end{align*}
where $\widehat{\A}$ denotes the image of $\A$ under the Gelfand transform, that is, $\widehat{\A} = \{ \widehat{f} : f \in \A \}$.  
The Shilov boundary for $\widehat{\A}$, which always exists due to the compactness of $\M(\A)$, will be denoted by $\partial \widehat{\A}$. Let us denote by $\partial \A$ the Shilov boundary for $\A$ only when it exists. 
%As any peak point for $\hat \A$ belongs to each boundary for $\hat \A$, the set $\rho\widehat{\A}$ is contained in $\partial \widehat{\A}$. 
In particular, when $\A = \A_{wu} (B_X), \A_u (B_X)$ or $\A_\infty (B_X)$, we will simply denote by $\rho \widehat{\A_{wu}}, \rho\widehat{\A_u}$ or $\rho\widehat{\A_{\infty}}$ (resp., $\partial \widehat{\A_{wu}}, \partial\widehat{\A_u}$ or $\partial \widehat{\A_{\infty}}$) the set of all peak points for $\widehat{\A}$ (resp., the Shilov boundary for $\widehat{\A}$) when the underlying Banach space is understood. 

\section{On the Shilov boundary for $\widehat{\A_u (B_X)}$ or $\widehat{\A_\infty (B_X)}$ }\label{section_1}

 \begin{proposition}\label{prop:general}
 Let $X$ be a Banach space and $\A = \A_u (B_X)$ or $\A_\infty (B_X)$. If $\phi$ is a strong boundary point for $\widehat{\A}$, then there exists a net $(x_\alpha)$ in $S_X$ such that $(\delta_{x_\alpha})$ converges in the Gelfand topology to $\phi$ in $\M (\A)$.  
 \end{proposition} 
 
 \begin{proof}
 Let $U$ be an open neighborhood of $\phi$ in $\M (\A)$. It suffices to show that there exists $x_0 \in S_X$ such that $\delta_{x_0} \in U$. Since $\phi$ is a strong boundary point, there exists $f \in \A$ with $\| f \| = 1$ such that $\phi (f) = 1$ and $\sup \{|\psi (f) | : \psi \in \M(\A) \setminus U \} < 1$. Notice that $\| f \| = \sup_{x \in S_X} |f(x)|$ where $f$ is understood as its extension to $\overline{B}_X$. Indeed, for a sequence $(x_n)$ in $B_X$ such that $|\delta_{x_n} (f) | = |f(x_n)| \rightarrow 1$, let $r_n = \|x_n\| < 1$ for each $n \in \N$. Consider the function 
$g_n : r_n^{-1} \overline{\mathbb{D}} \rightarrow \mathbb{C}$ given by 
\[
g_n (\lambda) = f(\lambda x_n) \,\text{  for all  }\, \lambda \in r_n^{-1} \overline{\mathbb{D}}.
\]
Note that $g_n$ is holomorphic on $r_n^{-1} \mathbb{D}$ and (uniformly) continuous on $r_n^{-1} \overline{\mathbb{D}}$. By the Maximum Modulus Principle, there exists $\lambda_n$ with $|\lambda_n| = r_n^{-1}$ such that $|g_n(\lambda_n)| \geq |g(1)|$, that is, $|f(\lambda_n x_n)| \geq |f(x_n)|$. Letting $y_n = \lambda_n x_n$ for each $n \in \N$, we have that $(y_n) \subset S_X$ and $|f(y_n)| \rightarrow \| f \|$; hence $\| f \| = \sup_{x \in S_X} |f(x)|$. This implies that there exists $x_0 \in S_X$ such that $\delta_{x_0} \in U$. 
 \end{proof} 
 
% \begin{remark}\label{remark:sphere}
%In the above Proposition \ref{prop:general}, if $\A=\A_u (B_X)$ or $\A_\infty (B_X)$, then we can actually choose a net $(x_\alpha)$ in $S_X$ such that $(\delta_{x_{\alpha}})$ converges in the Gelfand topology to $\phi$ in $\M(\A)$. To this end, we claim that each $f \in \A$ satisfies that $\| f \| = \sup_{x \in S_X} |f(x)|$, where $f$ is considered to be extended continuously to the closed unit ball $\overline{B}_X$. 
% If this claim is established, then we can take a sequence $(x_n) \subset S_X$ in the proof of Proposition \ref{prop:general}. For the proof of the claim, let $(x_n)$ be a sequence in $B_X$ such that $|f(x_n)| \rightarrow \|f\|$. For each $n \in \N$, let $r_n = \|x_n\| < 1$ and consider the function 
%$g_n : r_n^{-1} \overline{\mathbb{D}} \rightarrow \mathbb{C}$ given by 
%\[
%g_n (\lambda) = f(\lambda x_n) \,\text{  for all  }\, \lambda \in r_n^{-1} \overline{\mathbb{D}}.
%\]
%%Note that $g_n$ is holomorphic on $r_n^{-1} \mathbb{D}$ and (uniformly) continuous on $r_n^{-1} \overline{\mathbb{D}}$. By the Maximum Modulus Principle, there exists $\lambda_n$ with $|\lambda_n| = r_n^{-1}$ such that $|g_n(\lambda_n)| \geq |g(1)|$, that is, $|f(\lambda_n x_n)| \geq |f(x_n)|$. Letting $y_n = \lambda_n x_n$ for each $n \in \N$, we have that $(y_n) \subset S_X$ and $|f(y_n)| \rightarrow \| f \|$. It follows that $\| f \| = \sup_{x \in S_X} |f(x)|$.
%\end{remark} 

One of direct consequences of Proposition \ref{prop:general} is that the set of strong boundary points for $\widehat{\A}$ is contained in $\overline{\{\delta_x : x \in S_X\}}^{\,\sigma}$ where $\A = \A_\infty (B_X)$ or $\A_u (B_X)$; hence the Shilov boundary $\partial \widehat{\A}$ is contained also in $\overline{\{\delta_x : x \in S_X\}}^{\,\sigma}$ since the set of strong boundary points for $\widehat{\A}$ is dense in $\partial \widehat{\A}$ \cite[Corollary 4.3.7]{Dales}. Moreover, note that if there is a subset $\Gamma \subset \overline{B}_{X^{**}}$ satisfying that $\|  \widetilde{f} \| = \sup_{z \in \Gamma} |  \widetilde{f} (z) |$, where $ \widetilde{f}$ is the Aron-Berner extension of $f$, for all $f \in \A$, then 
\[
\| \widehat{f} \| = \| f \| = \|  \widetilde{f} \| = \sup_{z \in \Gamma} |  \widetilde{f} (z) | = \sup \{ | \widehat{f} (\delta_z) | : z \in \Gamma \}. 
\]
This simple observation yields the following. 

%Given a Banach space $X$ and $f \in \A = \A_u (B_X)$ or $\A_\infty (B_X)$, 

% As a matter of fact, the following can be shown. %For $x \in \overline{B}_{X^{**}}$, let us denote by $\delta_x$ the point evaluation homomorphism at $x$ in $\M (\A)$, where $\A  = \A_u (B_X)$ or $\A_\infty (B_X)$. That is, $\delta_x ( f) =  \widetilde{f} (x)$ for every $f \in \A$.  

\begin{proposition}\label{prop:boundary_norming}
Let $X$ be a Banach space and $\A = \A_u (B_X)$ or $\A_\infty (B_X)$. If $\Gamma \subset \overline{B}_{X^{**}}$ is a boundary for $\widetilde{\A} := \{  \widetilde{f} : f \in \A\}$, where $ \widetilde{f}$ is the Aron-Berner extension of $f$, then 
$\partial \widehat{\A}$ is contained in $\overline{\{\delta_z : z \in \Gamma \}}^{\,\sigma}$. 
\end{proposition}

As noticed in the above comment, for a uniform algebra, the set of strong boundary points is related to its Shilov boundary. Thus, in order to describe the Shilov boundary for $\widehat{\A}$, where $\A = \A_u (B_X)$ or $\A_\infty (B_X)$, it is natural to compare the behavior of strong boundary points for $\A$ and for $\widehat{\A}$. The following result shows that they have some relevance, and it will play an important role in the sequel. 
 
 \begin{theorem}\label{prop:strong_peak}
 Let $X$ be a Banach space and $\A = \A_u (B_X)$ or $\A_\infty (B_X)$. 
 \begin{enumerate} 
 \setlength\itemsep{0.4em}
 \item If $x \in S_X$ is a strong boundary point for $\A_a (B_X)$, then $\delta_x$ is a strong boundary point for $\widehat{\A}$. 
 \item If $x \in S_X$ is a strong peak point for $\A_a (B_X)$, then $\delta_x$ is a peak point for $\widehat{\A}$.
 \item If $\delta_x$ is a peak point for $\M (\A)$ for some $x \in S_X$, then $x$ is a peak point for $\A$. 
 \end{enumerate}
 \end{theorem} 
 
 \begin{proof}
(1): Let $x_0 \in S_X$ be a strong boundary point for $\A_a (B_X)$. 

\noindent \emph{STEP I}. If $\phi \in \M (\A)$ satisfies that $\pi (\phi) = x_0$, then $\phi = \delta_{x_0}$. 

To this end, let $g \in \A \setminus \{0\}$ be given. For $n \in \N$, choose $r_n > 0$ such that $| \widetilde{g}(z) -  \widetilde{g}(x_0)| < \frac{1}{n}$ for every $z \in \overline{B}_{X^{**}}$ with $\| z - x_0 \| < r_n$, where $ \widetilde{g}$ is the Aron-Berner extension of $g$. Since $x_0$ is a strong boundary point for $\A_a (B_X)$, there exists $f_n \in \A_a (B_X)$ such that $f_n (x_0) = \|f_n\| = 1$ and $\sup \{ |f_n (y)| : \| y - x_0\| \geq r_n \}  < 1$. It follows that there exists $M_n \in \N$ such that 
\[
\sup \left\{ \left| \frac{1 + f_n (y)}{2} \right|^{M_n} : \| y - x_0\| \geq r_n \right\} < \frac{1}{2n\|g\|}.
\]
Thus, we can conclude that 
\[
\left\| \left(\frac{1 + f_n }{2} \right)^{M_n}  (g- g(x_0)) \right\| \leq \frac{1}{n};
\]
hence 
\begin{equation}\label{eq:g-g(x_0)}
\left[ 1 - \left(\frac{1 + f_n }{2} \right)^{M_n} \right] (g- g(x_0)) \rightarrow g-g(x_0)
\end{equation} 
as $n \rightarrow \infty$. On the other hand, since  
\[
\phi \left( \frac{1 + f_n}{2} \right) = \frac{1 + \phi(f_n)}{2} = \frac{ 1+ f_n (x_0)}{2} = 1
\]
for each $n \in \N$, we have from \eqref{eq:g-g(x_0)} that $\phi (g) = g(x_0)$. Since $g$ is arbitrarily chosen, we conclude that $\phi = \delta_{x_0}$.  

\noindent \emph{STEP II}. $\delta_{x_0}$ is a strong boundary point for $\widehat{\A}$. 

Let $\eps >0$ and set $W := \{ \phi \in \M(\A) : |(\phi - \delta_{x_0}) (h_j) | < \eps \text{ for each } j =1,\ldots, N \}$ with $h_1,\ldots, h_N \in \A$. It suffices to prove that there exists $g \in \A$ such that $g(x_0) = \|g \| = 1$ and $|\psi (g) | < 1$ for every $\psi \in \M(\A) \setminus W$ thanks to the compactness of $\M(\A)$. Since each $h_j$ is continuous at $x_0$, we can choose $ r > 0$ small enough such that if $z \in \overline{B}_{X^{**}}$ satisfies $\| z - x_0 \| < r$, then $\delta_z \in W$. Take $n_0 \in \N$ so that ${n_0}^{-1} < r$ and for each $n \geq n_0$, consider $g_n \in \A_a (B_X)$ satisfying that $g_n (x_0) = \|g_n \| = 1$ and $\sup \{ |g_n (y)| : \| y - x_0\| \geq n^{-1} \} < 1$. 
Define 
\[
g := \sum_{n=n_0} \frac{1}{2^{n-n_0+1}} g_n
\]
and note that $g \in \A$ with $\|g \| = 1$. Note also that $g (x_0) = 1$. Now, for $\psi \in \M(\A) \setminus W$, if we assume that $| \psi (g) | =1$, then this implies that 
\[
1 = | \psi (g) | = \left| \sum_{n=n_0} \frac{1}{2^{n-n_0 +1}} g_n (\pi (\psi)) \right| \leq 
\sum_{n=n_0} \frac{1}{2^{n-n_0 +1}} \left| g_n (\pi (\psi)) \right| \leq 
\sum_{n=n_0} \frac{1}{2^{n-n_0 +1}} \leq 1;
\]
hence $|g_n (\pi (\psi))| = 1$ for all $n \geq n_0$. From this, we have that $\| \pi (\psi) - x_0 \| < n^{-1}$ for all $n \geq n_0$, which implies that $\pi (\psi ) = x_0$. By STEP I, we conclude that $\psi = \delta_{x_0}$. This contradicts to that $\psi \in \M (\A) \setminus W$; hence $|\psi (g)| < 1$. Thus, $g$ is a desired one. 

(2): Let $x \in S_X$ be a strong peak point for $\A_a (B_X)$ and let $\psi \in \M(A) \setminus \{\delta_x\}$ be given. Find $f \in \A_a (B_X)$ with $\| f \|=1$ which peaks strongly at $x$. Notice that $| \widehat{f} (\psi)| = |\psi (f)| = | \widetilde{f} (\pi (\psi)) |$, where $ \widetilde{f}$ is the Aron-Berner extension of $f$. 
If $| \widetilde{f} (\pi (\psi))| = 1$, then for a net $(x_\alpha)$ in $B_X$ converging weak-star to $\pi (\psi)$, we obtain that $\lim_\alpha |f(x_\alpha)| = | \widetilde{f} (\pi (\psi) )| = 1$. This implies that $(x_\alpha)$ must converge in norm to the point $x$; hence $\pi (\psi) = x$. Since $x$ is, in particular, a strong boundary point for $\A_a (B_X)$, by the above STEP I, we see that $\psi = \delta_x$. This is a contradiction. Thus, $|\widehat{f} (\psi)| < 1$ and $\widehat{f} \in \widehat{\A}$ peaks at $\delta_x$. 

(3): Suppose that $\delta_x$ is a peak point for some $x \in S_X$. Let $f \in \A$ satisfy that $f(x) = 1$ and $|\psi (f)| < 1$ for any $\psi$ in $\M (\A)$ which is different from $\delta_x$. As this implies that $|f(y)| < 1$ for every $y \in \overline{B}_X$ with $y \neq x$, we conclude that $f$ peaks at the point $x$. 
 \end{proof} 
 
Let us remark that STEP I in the proof of Theorem \ref{prop:strong_peak} generalizes \cite[Lemma 2.4]{ACGLM}, where it is observed that for any strong peak point $x$ for $\A_a (B_X)$, if $\phi \in \M(\A)$ satisfies that $\pi (\phi) = x$, then $\phi = \delta_x$.

From the above results, we can obtain that if $X$ is a finite dimensional Banach space and $\A = \A_u (B_X)$ (which coincides with $\A_\infty (B_X)$ and with $\A_a (B_X)$), then the Shilov boundary $\partial \widehat{\A}$ of $\widehat{\A}$ is essentially equal to the Shilov boundary for $\A$. Before stating this result, let us recall that for a Banach space $X$, $x \in \overline{B}_X$ is said to be a \emph{$\mathbb{C}$-extreme point} if $y \in X$ satisfies that $\| x + \lambda y \| \leq 1$ for any $\lambda \in \mathbb{C}$ with $|\lambda| = 1$, then $y=0$. 
We denote by $\Ext_{\mathbb{C}} (\overline{B}_X)$ the set of all $\mathbb{C}$-extreme points of the closed unit ball $\overline{B}_X$.

\begin{proposition}
Let $X$ be a finite dimensional Banach space and let $\A = \A_u (B_X)$. Then, 
\[
\{ \delta_z : z \in {\Ext_{\mathbb{C}} (\overline{B}_X)} \} \subseteq \rho \widehat{\A} \subseteq \partial \widehat{\A} \subseteq \left\{ \delta_z : z \in \overline{\Ext_{\mathbb{C}} (\overline{B}_X)} \right\}.
\]
In particular, $ \partial \widehat{\A} = \left\{ \delta_z : z \in \overline{\Ext_{\mathbb{C}} (\overline{B}_X)} \right\}$. 
%\begin{enumerate}
%\setlength\itemsep{0.4em}
%\item $\rho \widehat{\A}$ coincides with $\{ \delta_z : z \in {\Ext_{\mathbb{C}} (\overline{B}_X)} \}$, 
%\item $\partial \widehat{\A}$ coincides with $\{ \delta_z : z \in \overline{\Ext_{\mathbb{C}} (\overline{B}_X)} \}$. 
%the Shilov boundary for $\widehat{\A}$ coincides 
\end{proposition} 

\begin{proof}
Notice that the Shilov boundary for $\A$ is $\overline{ \Ext_{\mathbb{C}} (\overline{B}_X) }$ (see, for instance, \cite[Proposition 1.1]{CHL}). 
As $\overline{ \Ext_{\mathbb{C}} (\overline{B}_X) }$ is compact, one can check that $\{ \delta_z : z \in \overline{ \Ext_{\mathbb{C}} (\overline{B}_X)} \} $ is closed in $\M (\A)$. 
It follows from Proposition \ref{prop:boundary_norming} that the Shilov boundary for $\widehat{\A}$ is contained in $\{ \delta_z : z \in \overline{ \Ext_{\mathbb{C}} (\overline{B}_X)} \} $. On the other hand, $\Ext_{\mathbb{C}} (\overline{B}_X)$ is the same as the set of all strong peak points for $\A_a (B_X)$ \cite{Arenson}. Thus, by Theorem \ref{prop:strong_peak}, we have that $\{ \delta_z : z \in \overline{ \Ext_{\mathbb{C}} (\overline{B}_X)} \} $ is contained in $\rho \widehat{\A}$. 
\end{proof} 
  
In the above proof, let us point out that $\left\{ \delta_z : z \in \overline{\Ext_{\mathbb{C}} (\overline{B}_X)} \right\}$ is closed (hence compact) in $\M (\A)$ due to the relative compactness of $\Ext_{\mathbb{C}} (\overline{B}_X) \subseteq S_X$. For a Banach space $X$, we shall observe that this happens only when $X$ is finite dimensional. 

\begin{proposition}
Let $X$ be a Banach space and $\A = \A_u (B_X)$ or $\A_\infty (B_X)$. Then the set $\{\delta_x : x \in S_X \}$ is closed in $\M(\A)$ if and only if $X$ is finite dimensional. 
\end{proposition}

\begin{proof}
Suppose that $\{ \delta_x : x \in S_X\}$ is closed in $\M(\A)$. 

\noindent \emph{STEP I}. $X$ must have the Schur property. 

Assume to the contrary that $X$ does not have the Schur property. Then there is a sequence $(x_n)$ in $S_X$ such that $(x_n)$ converges in weak topology to $0$. Take a subnet $(\delta_{x_\alpha})$ which converges in the Gelfand topology to some $\phi$ in $\M(\A)$. Then $(x_\alpha)=(\pi (\delta_{x_\alpha}))$ converges weak-star to $\pi (\phi)$; hence $\pi (\phi) = 0$. It follows that $\phi \in \overline{\{\delta_x : x \in S_X \}}^{\,\sigma} \setminus \{\delta_x : x \in S_X \}$; hence the set $\{ \delta_x : x \in S_X\}$ is not closed. 

\noindent \emph{STEP II}. $X$ must be a finite dimensional space.  

Assume to the contrary that $X$ is an infinite dimensional Banach space. Since $X$ has the Schur property, by Rosenthal's $\ell_1$-theorem \cite{Rosenthal}, $X$ has an isomorphic copy of $\ell_1$. Let $T$ be a norm one isomorphism from $\ell_1$ into $X$ and denote by $(e_n)$ the canonical basis of $\ell_1$. For each $n \in \N$, let $x_n := T(e_n)$. Put $z_n = x_n/\|x_n\| \in S_X$ for each $n \in \N$ and take a subnet $(z_\alpha)$ so that $(\delta_{z_\alpha})$ converges in the Gelfand topology to some $\phi$ in $\M(\A)$. Since $\{ \delta_x : x \in S_X\}$ is closed, $\phi = \delta_u$ for some $u \in S_X$. In particular, this implies that $(z_\alpha)$ converges weakly to $u$ in $X$. 
As $\| x_\alpha \| \geq 1/ \|T^{-1}\|$ for each $\alpha$, passing to a subnet, we may assume that $\|x_\alpha\| \rightarrow r$ for some $r >0$. On the other hand, since $u \in T(\ell_1)$, we have that 
\[
\frac{1}{\|x_\alpha\|} e_\alpha = \frac{1}{\|x_\alpha\|} T^{-1} (x_\alpha) = T^{-1} (z_\alpha) \rightarrow T^{-1} (u)
\] 
weakly in $\ell_1$. On the other hand, $(\|x_\alpha\|^{-1} e_\alpha)$ converges weakly to $0$ in $\ell_1$, which implies that $T^{-1} (u)$ must be zero; hence $u$ is zero. This is a contradiction.
%which implies that $\pi (\phi) \neq 0$. Suppose that $\pi (\phi) \in \overline{B}_X$, then $(x_\alpha)$ converges weakly to $\pi (\phi)$ in $X$. This implies that $\pi (\phi)$ belongs to the image of the isomorphism $T$; hence we have that $T^{-1} (x_\alpha)$ converges weakly to $T^{-1} (\pi (\phi))$. In other words, $(e_\alpha)$ converges weakly to $T^{-1} (\pi (\phi))$ in $\ell_1$. However, it is clear that $(e_\alpha)$ converges weak-start to $0$ in $\ell_1$, which implies that $T^{-1} (\pi (\phi)) = 0$. It follows that $\pi (\phi) = 0$, which is a contradiction.

Conversely, assume that $X$ is a finite dimensional Banach space. Then $\{ \delta_x : x \in S_X \}$ is clearly closed in $\M (\A)$ since the unit sphere $S_X$ is compact. 
\end{proof}

Next, we present a result which will yield examples of Banach spaces $X$ where the Shilov boundary for $\A_u (B_X)$ or $\A_\infty (B_X)$ can be explicitly obtained. Recall that a point $x \in S_X$ is said to be a \emph{strongly exposed point} if there exists $x^* \in S_{X^*}$ such that $x^* (x) = 1$ and $(x_n)$ converges to $x$ in norm whenever $(\re x^* (x_n))$ converges to $1$. Notice that each strongly exposed point in $S_X$ is a strong peak point for $\A_a (B_X)$. Indeed, for a strongly exposed point $x \in S_X$, take $x^* \in S_{X^*}$ which strongly exposes at $x$. Then it is easy to check that $f:=\frac{1}{2}(1+x^*)$ strongly peak at $x$. %Recall that any point of the unit sphere of a locally uniformly convex Banach space $X$ is a strong peak point for $\A_a (B_X)$. Indeed, let $x \in S_X$. Take $x^* \in S_{X^*}$ so that $x^* (x) = 1$. Then it is easy to check that $f:=\frac{1}{2}(1+x^*)$ strongly peak at $x$. 
 
 \begin{proposition}\label{prop:LUR}
 Let $X$ be a Banach space such that every point in $S_X$ is a strongly exposed point and $\A = \A_u (B_X)$ or $\A_\infty (B_X)$. Then %
 \[
\{ \delta_x : x \in S_X\} \subseteq \rho \widehat{\A} \subseteq \partial \widehat{\A} \subseteq  \overline{\{ \delta_x : x \in S_X\}}^{\,\sigma}.
 \]
 In particular, $\partial \widehat{\A} = \overline{\{ \delta_x : x \in S_X\}}^{\,\sigma}$. 
 %the Shilov boundary $\widehat{\partial}_{\A}$ of $\widehat{\A}$ coincides with $\overline{\{ \delta_x : x \in S_X \}}^{\,\sigma}$.
 \end{proposition} 
 
 \begin{proof}
By the comment after Proposition \ref{prop:general}, it suffices to show that $\{ \delta_x : x \in S_X \}$ is contained in $\rho \widehat{\A}$. To this end, note that Theorem \ref{prop:strong_peak} combined with the fact that every point of $S_X$ is a strong peak point for $\A_a (B_X)$ implies that each $\delta_x$ with $x \in S_X$ is a peak point for $\widehat{\A}$. %In particular, we obtain that $\{ \delta_x : x \in S_X\}$ is contained in $\partial \widehat{\A_u}$ and complete the proof. 
 \end{proof} 
 
It can be easily checked that when $X$ is a locally uniformly convex Banach space, every point in $S_X$ is a strongly exposed point; hence a strong peak point for $\A_a (B_X)$. 

\begin{cor}
 Let $X$ be a locally uniformly convex Banach space and $\A = \A_u (B_X)$ or $\A_\infty (B_X)$. Then %
 \[
\{ \delta_x : x \in S_X\} \subseteq \rho \widehat{\A} \subseteq \partial \widehat{\A} \subseteq  \overline{\{ \delta_x : x \in S_X\}}^{\,\sigma}.
 \]
 In particular, $\partial \widehat{\A} = \overline{\{ \delta_x : x \in S_X\}}^{\,\sigma}$. 
\end{cor}

 Also, it is known \cite{AL} that any point of the unit sphere of $X=\ell_1$ or $X=d(w,1)$ for a decreasing sequence $w=(w_n)$ of positive real numbers with $w_1 = 1$ and $w \in c_0 \setminus \ell_1$, the complex Lorentz sequence space, is a strong peak point for $\A_a (B_X)$. So, arguing in the same way as in Proposition \ref{prop:LUR}, we have the following. 
 
  \begin{proposition}
 Let $X = \ell_1$ or $d(w,1)$ for a decreasing sequence $w=(w_n)$ of positive real numbers with $w_1 = 1$ and $w \in c_0 \setminus \ell_1$,  and $\A = \A_u (B_X)$ or $\A_\infty (B_X)$. Then %
 \[
\{ \delta_x : x \in S_X\} \subseteq \rho \widehat{\A} \subseteq \partial \widehat{\A} \subseteq  \overline{\{ \delta_x : x \in S_X\}}^{\,\sigma}.
 \]
 In particular, $\partial \widehat{\A} = \overline{\{ \delta_x : x \in S_X\}}^{\,\sigma}$. 
 \end{proposition} 
 
 From a careful examination of the proofs of \cite[Proposition 3.2 and Corollary 3.4]{CHL}, we can see that if $X$ is a locally $c$-uniformly convex order continuous sequence space, then a finitely supported point $x_0$ in $S_X$ is a strong peak point for $\A_a (B_X)$. Indeed, it is shown that if $x_0$ is finitely supported, then there exists $g \in \A_u (B_Y)$ and $P$ a projection from $X$ onto $Y$ such that $g \circ P$ strongly peak at $x_0$, where $Y$ is a finite dimensional subspace of $X$. Note that $\A_u (B_Y) = \A_a (B_Y)$ since $Y$ is finite dimensional; hence $g \circ P$ belongs to $\A_a (B_X)$. 

\begin{proposition}
Let $X$ be a locally $c$-uniformly convex order continuous sequence space and $\A = \A_u (B_X)$ or $\A_\infty (B_X)$. Then %
 \[
\{ \delta_x : x  \text{ is finitely supported in } S_X \} \subseteq \rho \widehat{\A} \subseteq \partial \widehat{\A} \subseteq  \overline{\{ \delta_x : x \in S_X\}}^{\,\sigma}.
 \]
In particular, $\partial \widehat{\A} = \overline{\{ \delta_x : x \in S_X\}}^{\,\sigma}$. 
\end{proposition} 

\begin{proof}
It is shown in \cite[Theorem 3.5]{CHL} that $S_X$ is the Shilov boundary for $\A$. Thus, $\partial \widehat{\A} \subseteq \overline{\{ \delta_x : x \in S_X \}}^{\, \sigma}$ by Proposition \ref{prop:boundary_norming}. As any finitely supported point $x$ in $S_X$ is a strong peak point for $\A_a (B_X)$,  we observe from Theorem \ref{prop:strong_peak} that 
\[
\{ \delta_x : x  \text{ is finitely supported in } S_X \} \subseteq  \rho \widehat{\A}. 
\]
As the set of finitely supported elements in $S_X$ forms a dense subset of $S_X$, a density argument implies that $\{ \delta_x : x \in S_X\} \subseteq  \partial \widehat{\A}$, which completes the proof. 
\end{proof}

On the other hand, it is known \cite{Glob1} that for a Banach space, every peak point for $\A_u (B_X)$ or $\A_\infty (B_X)$ is a $\mathbb{C}$-extreme point of $\overline{B}_X$. Since $\overline{B}_{c_0}$ has no $\mathbb{C}$-extreme point; hence, Theorem \ref{prop:strong_peak} cannot be applied to compute the Shilov boundary for $\widehat{\A_u (B_{c_0})}$. Nevertheless, we can describe the Shilov boundary for $\widehat{\A_u (B_{c_0})}$ by using a well known result of Littlewood-Bogdanowicz-Pelczynski \cite[Proposition 1.59]{Dineen} which asserts that any bounded $m$-homogeneous polynomial on $c_0$ can be approximated uniformly on $B_{c_0}$ by $m$-homogeneous finite type polynomials. Notice from this that $\M (\A_u (B_{c_0}))$ is actually nothing but $\{ \delta_x : x \in \overline{B}_{\ell_\infty}\}$.

 \begin{proposition}\label{prop:A_u_c_0}
The set of peak points for $\widehat{\A_u (B_{c_0})}$ and the Shilov boundary for $\widehat{\A_u (B_{c_0})}$ both coincide with ${\{ \delta_z : z \in \mathbb{T}^{\mathbb{N}} \}}$. 
 \end{proposition} 
 
 \begin{proof}
It is shown in \cite[Theorem 2]{ACLP} that $\mathbb{T}^{\mathbb{N}}$ is a closed boundary for $\A_\infty (B_{\ell_\infty})$; hence $\{\delta_z : z \in \mathbb{T}^{\mathbb{N}} \}$ is a boundary for $\widehat{\A_u}$. By Proposition \ref{prop:boundary_norming}, it follows that $\partial \widehat{\A_u}$ is contained in $\overline{\{ \delta_z : z \in \mathbb{T}^{\mathbb{N}} \}}^{\, \sigma}$. 

%Notice that $\{ \delta_z : z \in \mathbb{T}^{\mathbb{N}} \}$ is actually closed in $\M (\A_u (B_{c_0}))$. Indeed, if a net $(z_\alpha)$ in $\mathbb{T}^\N$ satisfies that $(\delta_{z_\alpha})$ converges to some $\phi \in \M(\A_u (B_{c_0}))$, then $\phi = \delta_z$ for some $z$ in $\overline{B}_{\ell_\infty}$.  Since $(z_\alpha)$ converges weak-star to $z$, we see that $z$ belongs to $\mathbb{T}^\N$. Consequently, $\partial \widehat{\A_u}$ is a subset of $\{ \delta_z : z \in \mathbb{T}^{\mathbb{N}} \}$.    

On the other hand, if $z_0 \in \mathbb{T}^{\mathbb{N}}$, then there is $x^* \in \ell_1$ such that $f = \frac{1+x^*}{2}$ satisfies that $ \widetilde{f} (z_0) = 1$ and $| \widetilde{f} (w)| < 1$ whenever $w \in \overline{B}_{\ell_\infty}$ is different from $z_0$, where $ \widetilde{f}$ is the Aron-Berner extension of $f$ (see \cite[Proposition 3.6]{ADLM}). Notice that $\widehat{f}$ attains its norm at a strong boundary point $\phi$ in $\M (\A)$ \cite[Proposition 4.3.4]{Dales}. Since $\phi (f) =  \widetilde{f} (\pi (\phi))$, we obtain that $\pi (\phi)$ must be $z_0$. 
The fact that $\M (\A_u (B_{c_0}))$ coincides with $\{ \delta_z : z \in \overline{B}_{\ell_\infty}\}$ implies that $\phi$ must be $\delta_{z_0}$; hence $\delta_{z_0}$ is a strong boundary point. Since $\ell_1$ is separable, the spectrum $\M (\A_u (B_{c_0}))$ is compact metrizable which implies that $\delta_{z_0}$ is indeed a peak point for $\widehat{\A_u (B_{c_0})}$ \cite[Corollary 4.3.7]{Dales}. It follows that 
\[
\{ \delta_z : z \in \mathbb{T}^{\mathbb{N}} \} \subseteq \rho \widehat{\A_u} \subseteq \partial \widehat{\A_u} \subseteq \overline{\{ \delta_z : z \in \mathbb{T}^{\mathbb{N}} \}}^{\, \sigma}. 
\]

To complete the proof, we check that $\{ \delta_z : z \in \mathbb{T}^{\mathbb{N}} \}$ is actually closed in $\M (\A_u (B_{c_0}))$. Indeed, if a net $(z_\alpha)$ in $\mathbb{T}^\N$ satisfies that $(\delta_{z_\alpha})$ converges to some $\phi \in \M(\A_u (B_{c_0}))$, then $\phi = \delta_z$ for some $z$ in $\overline{B}_{\ell_\infty}$.  Since $(z_\alpha)$ converges weak-star to $z$, we see that $z$ belongs to $\mathbb{T}^\N$. 
 \end{proof}

%Furthermore, we can also describe the Shilov boundary for $\widehat{\A_u (B_{\ell_\infty})}$. 

By arguing similarly as in the proof of Proposition \ref{prop:A_u_c_0}, we can obtain the description of the Shilov boundary for $\widehat{\A_u (B_{C(K)})}$ for certain compact Hausdorff space $K$. 
It is worth to remark that Choi et al. \cite{CGKM} first showed that there is no minimal closed boundary for $\A_u (B_{C(K)})$ when $K$ is infinite and scattered. Afterwards, it is generalized by Acosta \cite{Acosta} that if $K$ is an arbitrary infinite compact Hausdorff space, then $\A_u (B_{C(K)})$ or $\A_\infty (B_{C(K)})$ has no minimal closed boundary. Moreover, Acosta and Louren\c co observed in \cite{AL} that there are no strong peak points for $\A_u (B_{C(K)})$ if $K$ is infinite, and the set of peak points for $\A_u (B_{C(K)})$ is the set of extreme points of $B_{C(K)}$ if $K$ is separable.

Note that a compact Hausdorff space $K$ is countable if and only if $K$ is scattered and metrizable (see, for instance, \cite[Lemma 14.21]{FHHMPZ}). 
We present a description of the Shilov boundary for $\widehat{\A_u (B_{C(K)})}$ for an infinite countable compact Hausdorff space $K$. 
To this end, let us recall the following lemma which is already well known to experts. For the proof, we refer \cite{Aron} or \cite{Gamelin}. As usual, we denote by $\Pol(^N X)$ the Banach space of all $N$-homogenous polynomials on a Banach space $X$ and $\Pol_{wu}(^N X)$ the subspace of $\Pol(^N X)$ of all elements which are weakly uniformly continuous on bounded subsets of $X$. 

\begin{lemma}\label{lem:weak_unif_conti}
Let $X$ be a Banach space and $f \in \A_u (B_X)$. If $\sum_N P_N$ is the Taylor series expansion of $f$ at $0$, then $f \in \A_{wu} (B_X)$ if and only if $P_N \in \mathcal{P}_{wu}(^N X)$ for each $N \in \mathbb{N}$. 
\end{lemma}

\begin{proposition}\label{prop:AP_wu_spectrum}
Let $X$ be a Banach space whose dual has the approximation property. Then $\Pol (^N X) = \Pol_{wu} (^N X)$ for every $N \in \N$ if and only if $\M (\A_u (B_X)) = \{ \delta_z : z \in \overline{B}_{X^{**}}\}$. %(CHECK!) 
\end{proposition} 

\begin{proof}
We first prove the `only if' part. Let $\phi \in \M(\A_u (B_X))$, $f \in \A_u (B_X)$ and $\eps >0$ be given. Using Lemma \ref{lem:weak_unif_conti}, we can find a continuous polynomial $P$ on $X$ which is weakly uniformly continuous on bounded sets such that $\| f - P \| < \frac{\eps}{4}$. Since $X^*$ has the approximation property, notice from \cite[Proposition 2.8]{Dineen} that there exists a finite type polynomial $R$ such that $\| P - R \| < \frac{\eps}{4}$. Thus, we have that 
\[
|\phi(f) - \delta_{\pi(\phi)} (f) | \leq |\phi (f) - \phi (R)| + |  \widetilde{R} (\pi (\phi)) -  \widetilde{f}(\pi(\phi))| < \frac{\eps}{2} + \frac{\eps}{2} = \eps. 
\]
Since $\eps >0$ was arbitrary, we conclude that $\phi(f) = \delta_{\pi(\phi)} (f)$, so $\phi = \delta_{\pi (\phi)}$.  

Conversely, suppose that $\M (\A_u (B_X)) = \{ \delta_z : z \in \overline{B}_{X^{**}}\}$. Pick $P \in \Pol (^N X)$ and a bounded net $(x_\alpha)$ in $X$ converging weakly to some $x_0 \in X$. By homogeneity of $P$, we may assume that $\|x_\alpha \| \leq 1$ for every $\alpha$. We claim that $(P(x_\alpha))$ converges to $\widetilde{P}(z)$, where $\widetilde{P}$ is the Aron-Berner extension of $P$. If not, there exist $\eps >0$ and a subnet $(x_\beta)$ such that $|P(x_\beta)-\widetilde{P}(z)| \geq \eps$ for every $\beta$. Passing to a subnet, we may assume that $(\delta_{x_\beta})$ converges to some $\phi \in \M ( \A_u (B_X))$. Indeed, $\phi$ is $\delta_{z_0}$ for some $z_0 \in \overline{B}_{X^{**}}$ by our assumption and this implies that $(x_\beta)$ converges weak-star to $z_0$. Therefore, $z_0 = x_0$ and $\widetilde{P}(x_0)=\delta_{x_0} (P)=\lim_\beta \delta_{x_\beta} (P) =\lim_\beta P(x_\beta)$. This contradicts to that $|P(x_\beta)-\widetilde{P}(z)| \geq \eps$ for every $\beta$; hence we conclude that $(P(x_\alpha))$ converges to $\widetilde{P}(z)$. Thus, $P$ is weakly continuous on bounded subsets on $X$; hence $P \in \Pol_{wu}(^N X)$ \cite[Theorem 2.9]{AHV}. 
\end{proof}

\begin{proposition}\label{prop:C(K)}
Let $K$ be an infinite countable compact Hausdorff space. Then the set of peak points for $\widehat{\A_u (B_{C(K)})}$ and the Shilov boundary for $\widehat{\A_u (B_{C(K)})}$ both coincide with ${\{ \delta_u : u \in C(K)^{**}, |u(t)|=1, \, \forall t \in K\}}$. 
\end{proposition}

\begin{proof}
Recall that $C(S)$ has the Dunford-Pettis property for any compact Hausdorff space $S$ \cite[Theorem 5.4.5]{Albiac-Kalton}. For an infinite compact Hausdorff scattered space $K$, it is known \cite{PS} that $C(K)$ does not have an isomorphic copy of $\ell_1$. It follows from \cite[Corollary 2.37]{Dineen} that $\Pol(^N C(K)) = \Pol_{wu} (^N C(K))$ for every $N \in \N$. On the other hand, as $K$ is scattered, $C(K)^*$ is isometric to $\ell_1 (K)$, which has the approximation property (see, for instance, \cite[Theorem 14.24]{FHHMPZ}). Thus, we can apply Proposition \ref{prop:AP_wu_spectrum} to conclude that $\M (\A_u (B_{C(K)}))= \{ \delta_z : z \in \overline{B}_{C(K)^{**}}\}$. 
Let us write $K= \{ k_n : n \in \N\}$, and let $u \in C(K)^{**}$ be given so that $|u(t)|=1$ for every $t \in K$. Consider $T \in  \A_u (B_{C(K)})$ given by 
\[
T(v) = \frac{1}{2}\left(1 + \sum_{n=1}^\infty \frac{\overline{u(k_n)}}{2^n} v(k_n)\right)
\]
for each $v \in C(K)$. Notice that $|\widetilde{T} (z)| < 1$ for every $z \in \overline{B}_{C(K)^{**}} = \overline{B}_{\ell_\infty (K)}$ which is different from $u$. 
As a matter of fact, if $| \widetilde{T}(z)| = 1$ for some $z \in \overline{B}_{C(K)^{**}}$, then $z(k_n) = u(k_n)$ for every $n \in \N$; hence $z=u$. Note from \cite[Proposition 4.3.4]{Dales} that there exists a strong boundary point $\phi$ for $\widehat{\A_u (B_{C(K)})}$ in $\M(\A_u (B_{C(K)}))$ so that $|\widehat{T} (\phi)| = 1$. As observed, $\phi$ must be $\delta_u$ and it follows that $\delta_u=\phi$ is a strong boundary point. As $\M (\A_u (B_{C(K)}))$ is compact metrizable, $\delta_u$ is then a peak point. Hence, 
\[
\{ \delta_u : u \in C(K)^{**}, |u(t)|=1, \, \forall t \in K \} \subseteq \rho \widehat{\A_u (B_{C(K)})}.
\]

Conversely, by a result of M.D. Acosta \cite[Theorem 4.2]{Acosta}, the set of extreme points in the unit ball of $C(K)$ is a boundary for $\A_u (B_{C(K)})$. Thus, we have that 
\begin{align*}
\partial \widehat{\A_u (B_{C(K)})} &\subseteq \overline{\{ \delta_u : u \in C(K), |u(t)|=1, \, \forall t \in K\}}^{\, \sigma}
\end{align*} 
from Proposition \ref{prop:boundary_norming} and the fact that the set of extreme points in $B_{C(K)}$ coincides with $\{ u \in C(K) : |u(t)|=1, \, \forall t \in K\}$. Because of the following inclusion 
\begin{align*}
\overline{\{ \delta_u : u \in C(K), |u(t)|=1, \, \forall t \in K\}}^{\, \sigma} \subseteq \{ \delta_u : u \in C(K)^{**}, |u(t)|=1, \, \forall t \in K\},
\end{align*}
we complete the proof.
\end{proof}

%Note that the Stone-\v{C}ech compactification $\beta \N$ of $\N$ is neither scattered nor metrizable, thus Proposition \ref{prop:C(K)} cannot be applied to the case $\ell_\infty = C(\beta \N)$. 

The last example in this section concerns the Banach space $C_1 (H)$ of trace class operators on a complex Hilbert space $H$. Recall that an operator $T$ on $H$ is called a \emph{trace class operator} if there are orthonormal sequences $(e_n)$ and $(f_n)$ in $H$ such that $T x = \sum_{n=1}^\infty \lambda_n \langle x, e_n \rangle f_n$ for every $x \in H$ and $(\lambda_n)$ in $\ell_1$. In this case, the operator norm of $T$ is given by $\| T\| =\sum_{n=1}^\infty |\lambda_n|$.
For an orthonormal basis $\{ e_i : i \in I\}$ of $H$ and $F \subset I$, let us denote by $\Pi_F$ the norm one projection on $C_1(H)$ given by 
\[
\Pi_F (T) = P_F \circ T \circ P_F
\]
where $P_F (x) = \sum_{i \in F} \langle x, e_i \rangle e_i$ for every $x \in H$. 
It is observed in \cite[Theorem 4.1]{AcoLou} that the Shilov boundaries for $\A_u (B_{C_1 (H)})$ and $\A_\infty (B_{C_1 (H)})$ both exist and coincide. 
From a careful analysis of the proof of \cite[Theorem 4.1]{AcoLou}, we describe the Shilov boundary for $\widehat{\A}$ for $\A = \A_u (B_{C_1 (H)})$ or $\A_\infty (B_{C_1 (H)})$ as follows. 

\begin{proposition}
Let $H$ be a complex Hilbert space and let $\A = \A_u (B_{C_1 (H)})$ or $\A_\infty (B_{C_1 (H)})$. Then 
\[
\{ \delta_T : \, T \in \Gamma \} \subseteq \rho \widehat{\A} \subseteq \partial \widehat{\A} = \overline{\{\delta_T : T \in \partial \A  \}}^{\, \sigma},
\]
where
\begin{align*}\label{Pi_F_boundary}
\Gamma := \{ T \in B_{\Pi_F (C_1 (H))} : \, &T \text{ is a peak point for } \A_\infty (B_{\Pi_F (C_1 (H))}), \, F \text { is a finite subset of } I\}.
\end{align*} 
In particular, $\partial \widehat{\A} =\overline{\{\delta_T : T \in \partial \A  \}}^{\, \sigma}$. 
\end{proposition}  

\begin{proof}
It is proved in \cite[Theorem 4.1]{AcoLou} that for a finite subset $F \subset I$, if $T$ is a peak point for $\A_\infty (B_{\Pi_F (C_1 (H))})$ with peak function $g \in \A_\infty (B_{\Pi_F (C_1 (H))})$, then $g \circ \Pi_F \in \A_\infty (B_{C_1 (H)})$ strongly peaks at $T$. 
As $\Pi_F (C_1 (H))$ is finite dimensional, $g$ actually belongs to $\A_a (B_{\Pi_F (C_1 (H))})$; hence $g \circ \Pi_F$ belongs to $\A_a (B_{C_1 (H)})$. By Theorem \ref{prop:strong_peak}, it follows that $\{ \delta_T : \, T \in \Gamma \} \subseteq \rho \widehat{\A}$. Consequently, 
\[
\{ \delta_T : \, T \in \Gamma \} \subseteq \rho \widehat{\A} \subseteq \partial \widehat{\A} %\subseteq \overline{\{\delta_T : T \in \partial \A  \}}^{\, \sigma}. 
\]
On the other hand, it is observed in the same result in \cite{AcoLou} that $\overline{\Gamma}$ is the Shilov boundary for $\A$, that is, $ \overline{\Gamma}= \partial \A$. 
So, Proposition \ref{prop:boundary_norming} yields that $\partial \widehat{\A} \subseteq \overline{\{\delta_T : T \in \partial \A \}}^{\, \sigma}$.
%Moreover, it is observed in \cite[Theorem 4.1]{AcoLou} that for a finite subset $F \subset I$, 
%\begin{align*}
%\{ T \in B_{\Pi_F (C_1 (H))} : \, &T \text{ is a peak point for } \A_\infty (B_{\Pi_F (C_1 (H))}), \, F \text { is a finite subset of } I\} \\
%&\subseteq \{ T \in S_{C_1 (H)} : T \text{ is a strong peak point for } \A_u (B_{C_1 (H)})\}. 
%\end{align*} 
To see that $\partial \widehat{\A} = \overline{\{\delta_T : T \in \partial \A \}}^{\, \sigma}$, notice that if $T \in \partial \A$, then there exists a net $(T_\alpha)$ in $\Gamma$ so that $(\delta_{T_\alpha})$ converges to $\delta_T$. This shows that $\{ \delta_T : T \in \partial \A \} \subseteq \partial \widehat{\A}$. 
\end{proof}

We finish this section by giving a possible application of our result. Recall from \cite[Theorem 4.3]{ACG91} that for an infinite dimensional Banach space $X$, the closure of $S_X$ in the $H_b$-topology includes the closed unit ball $\overline{B}_{X^{**}}$, where $H_b$-topology is the weak topology on the spectrum of the algebra $H_b$ of complex-valued entire functions on $X$ determined by elements in $H_b$. This implies that, in particular, for any $z \in {B}_{X}$, there is a net $(x_\alpha)$ in $S_X$ such that $f(x_\alpha) \rightarrow \widetilde{f} (z)$ for every $f \in \A_u (B_X)$ where $\widetilde{f}$ is the Aron-Berner extension of $f$. Thus, we can observe that
\[
\overline{\{ \delta_x : x \in S_X \}}^{\, \sigma} = \overline{\{ \delta_x : x  \in B_X \}}^{\, \sigma} \text{ in } \M(\A_u (B_X)). 
\]
On the other hand, we observed that if $X$ is (1) a locally uniformly convex Banach space, (2) $\ell_1$, (3) $d(w,1)$, or (4) a locally $c$-uniformly convex order continuous sequence space, then the Shilov boundary $\partial \widehat{\A_u}$ for $\widehat{\A_u (B_X)}$ coincides with $\overline{\{\delta_x : x \in S_X \}}^{\,\sigma}$. Therefore, if $X$ is one of (1)-(4) with infinite dimension, then we have that 
\[
\partial \widehat{\A_u} = \overline{\{ \delta_x : x  \in B_X \}}^{\, \sigma}. 
\]
Recall from \cite[Section 2]{CGMS} that we say that the \emph{Corona theorem} holds for $\A_u (B_X)$  when $\M(\A_u (B_X))$ coincides with $\overline{ \{ \delta_x : x \in B_X\}}^{\,\sigma}$. Consequently, we have the following observation which shows the the Corona problem for $\A_u (B_X)$ can be restated in terms of the Shilov boundary for $\widehat{\A_u (B_X)}$ in some cases.  

\begin{theorem}
Let $X$ be one of the above (1)-(4) Banach spaces. Then the Corona theorem for $\A_u (B_X)$ is equivalent to that $\M (\A_u (B_X))$ coincides with its Shilov boundary $\partial \widehat{\A_u}$. 
\end{theorem}

\section{On the Shilov boundary for $\widehat{\A_{wu} (B_X)}$}\label{sec:wu}

In this section, we show that the Shilov boundary for $\widehat{\A_{wu} (B_X)}$ can be precisely described for a large class of Banach spaces $X$. Note that for each $f \in \A_{wu} (B_X)$, the Aron-Berner extension $ \widetilde{f}$ of $f$ belongs to $\A_{w^* u} (B_{X^{**}})$, where $\A_{w^* u} (B_{X^{**}})$ denotes the Banach algebra of all bounded holomorphic functions on $B_{X^{**}}$ which are weak-star uniformly continuous on $B_{X^{**}}$ \cite[Proposition 2.1]{AAM}. Indeed, $\A_{wu} (B_X)$ is isometrically isomorphic to $\A_{w^* u} (B_{X^{**}})$ via the mapping $\Phi : f \mapsto  \widetilde{f}$. Observe that $\Phi^* \vert_{\M (\A_{w^* u} (B_{X^{**}}) )}$ is a homeomorphism between $\M (\A_{w^* u} (B_{X^{**}}))$ and $\M (\A_{w u} (B_{X}))$ which sends point evaluations to point evaluations. 
From this observation, we obtain the following.

\begin{proposition}\label{prop:same_boundary}
Let $X$ be a Banach space. Then the Shilov boundary for $\widehat{\A_{wu} (B_X)}$ is homeomorphic to that of $\widehat{\A_{w^* u} (B_{X^{**}})}$ and the set of peak points for $\widehat{\A_{wu} (B_X)}$ is homeomorphic to that of $\widehat{\A_{w^* u} (B_{X^{**}})}$.
\end{proposition} 

\begin{proof}
Let $\Gamma \subset \M (\A_{w^* u} (B_{X^{**}}) )$ be a closed boundary for $\widehat{\A_{w^* u} (B_{X^{**}})}$. 
Notice that for fixed $f \in \A_{wu} (B_X)$, we have that 
\begin{align*}
\| f \| = \|  \widetilde{f} \| = \sup \{ | \phi (\Phi (f)) | : \phi \in \Gamma \} &= \sup \{ | \Phi^* (\phi) ( f)| : \phi \in \Gamma\} \\
&= \sup \{ | \psi ( f) | : \psi \in \Upsilon (\Gamma) \},
\end{align*} 
where $\Upsilon := \Phi^* \vert_{\M (\A_{w^* u} (B_{X^{**}}) )} : \M (\A_{w^* u} (B_{X^{**}})) \rightarrow \M (\A_{w u} (B_{X}))$. This implies that $\Upsilon (\Gamma) \subset \M(\A_{wu} (B_X))$ is a closed boundary for $\widehat{\A_{wu} (B_X)}$. Similarly, one can see that if $\Lambda \subset \M (\A_{wu} (B_X))$ is a closed boundary for $\widehat{\A_{wu} (B_X)}$, then $\Upsilon^{-1} (\Lambda)$ is a closed boundary for $\widehat{\A_{w^* u} (B_{X^{**}})}$. %From this, we can see that if $\Gamma \subset \M (\A_{w^* u} (B_{X^{**}}) )$ is the Shilov boundary for $\widehat{\A_{w^* u} (B_{X^{**}})}$, then $\Upsilon (\Gamma)$ is the Shilov boundary for $\widehat{\A_{wu} (B_X)}$. 

Finally, note that if $\widehat{f} \in \widehat{\A_{w^* u} (B_{X^{**}})}$ peaks at $\phi \in \M(\A_{w^* u} (B_{X^{**}}) )$, then $\widehat{\Phi^{-1} (f)}$ peaks at $\Upsilon (\phi)$. This implies that $\Upsilon (\rho \widehat{\A_{w^* u} } ) \subseteq \rho \widehat{\A_{wu}}$. Similarly, one can show that $\Upsilon^{-1} (\rho \widehat{\A_{wu}} ) \subseteq \rho \widehat{\A_{w^* u} }$. 
\end{proof} 

%Viewing $\A_{w^* u} (B_{X^{**}})$ as a uniform algebra on $(B_{X^{**}}, w(X^{**}, X^*))$, where $w(X^{**}, X^*)$ denotes the weak-star topology on $X^{**}$, 
Let us denote by $P({B}_{X^{**}})$ the closed subalgebra of $\A_{w^{*} u} (B_{X^{**}})$ generated by constants and the restrictions to $B_{X^{**}}$ of the elements of $X^*$. %It is known that the set of strong boundary points of $\overline{B}_{X^{**}}$ for $P(B_{X^{**}})$ coincides with $\Ext_{\mathbb{C}} (\overline{B}_{X^{**}})$, the set of all complex extreme points of $\overline{B}_{X^{**}}$ due to E.L. Arenson \cite{Arenson}.  

%From this point of view, it is observed in \cite[Proposition 2.3]{AAM} that $\Ext_{\mathbb{C}} (\overline{B}_{X^{**}})$ is the minimal subset of $\overline{B}_{X^{**}}$ satisfying that $\| f \| = \max \{ |f(z)| :z \in \overline{B}_{X^{**}} \}$ for all $f \in \A_{w^* u } (B_{X^{**}})$. 

\begin{theorem}\label{prop:separable_dual}
Let $X$ be a Banach space. Then, 
 \[
\{ \delta_z : z \in \Ext_{\mathbb{C}} (\overline{B}_{X^{**}}) \} \subseteq \partial \widehat{\A_{w^* u}}. 
 \]
If, in addition, $X^*$ is separable, then 
 \[
\{ \delta_z : z \in \Ext_{\mathbb{C}} (\overline{B}_{X^{**}}) \} \subseteq \rho \widehat{\A_{w^* u}} \subseteq \partial \widehat{\A_{w^* u}}  \subseteq  \overline{\{ \delta_z : z \in \Ext_{\mathbb{C}} (\overline{B}_{X^{**}}) \} }^{\,\sigma}.
 \]
In this case, the Shilov boundary for $\widehat{\A_{w^* u} (B_{X^{**}}) }$ coincides with $\overline{ \{ \delta_z : z \in \Ext_{\mathbb{C}} (\overline{B}_{X^{**}}) \} }^{\, \sigma}$.
\end{theorem} 

\begin{proof}
Due to E.L. Arenson \cite{Arenson}, it is known that $\Ext_{\mathbb{C}} (\overline{B}_{X^{**}})$ is the set of strong boundary points for $P(B_{X^{**}})$. Thus, the same argument used in Theorem \ref{prop:strong_peak} shows that the $\{ \delta_z : z \in \Ext_{\mathbb{C}} (\overline{B}_{X^{**}}) \}$ is contained in the set of strong boundary points for $\widehat{\A_{w^* u } (B_{X^{**}})}$; hence $\{ \delta_z : z \in \Ext_{\mathbb{C}} (\overline{B}_{X^{**}}) \}$ is contained in $\partial \widehat{\A_{w^* u}}$.

Suppose that $X^*$ is a separable Banach space. Then the domain $(B_{X^{**}}, w(X^{**}, X^*) )$ is compact metrizable. From this, we can deduce that $\Ext_{\mathbb{C}} (\overline{B}_{X^{**}})$ is the set of all strong peak points for $P({B}_{X^{**}})$. It follows from Theorem \ref{prop:strong_peak} that $\{ \delta_z : z \in \Ext_{\mathbb{C}} (\overline{B}_{X^{**}}) \}$ is contained in $\rho \widehat{\A_{w^* u}}$.
Moreover, it is observed in \cite[Proposition 2.3]{AAM} that $\Ext_{\mathbb{C}} (\overline{B}_{X^{**}})$ is a minimal subset of $B_{X^{**}}$ satisfying that $\| f \| = \max \{ |f(z)| :z \in \Ext_{\mathbb{C}} (\overline{B}_{X^{**}}) \}$ for all $f \in \A_{w^* u } (B_{X^{**}})$. Thus, Proposition \ref{prop:boundary_norming} implies that $\partial \widehat{\A_{w^* u}}  \subseteq  \overline{\{ \delta_z : z \in \Ext_{\mathbb{C}} (\overline{B}_{X^{**}}) \} }^{\,\sigma}$.  
\end{proof}

\begin{cor}
Let $X$ be a Banach space. Then, 
 \[
\{ \delta_z : z \in \Ext_{\mathbb{C}} (\overline{B}_{X^{**}}) \} \subseteq \partial \widehat{\A_{w u}}.  \]
If, in addition, $X^*$ is separable, then 
 \[
\{ \delta_z : z \in \Ext_{\mathbb{C}} (\overline{B}_{X^{**}}) \} \subseteq \rho \widehat{\A_{w u}} \subseteq \partial \widehat{\A_{w u}}  \subseteq  \overline{\{ \delta_z : z \in \Ext_{\mathbb{C}} (\overline{B}_{X^{**}}) \} }^{\,\sigma}.
 \]
In this case, the Shilov boundary for $\widehat{\A_{w u} (B_{X}) }$ coincides with $\overline{ \{ \delta_z : z \in \Ext_{\mathbb{C}} (\overline{B}_{X^{**}}) \} }^{\, \sigma}$.
\end{cor} 

\begin{proof}
Notice that $\Upsilon ( \{ \delta_z : z \in \Omega \}) = \{ \delta_z : z \in \Omega \}$ for a subset $\Omega$ of $\overline{B}_{X^{**}}$, where $\Upsilon$ is the homeomorphism in Proposition \ref{prop:same_boundary} and the former $\delta_z$ is understood as a member of $\M(\A_{w^* u} (B_{X^{**}}) )$ while the latter $\delta_z$ is a member of $\M(\A_{w u} (B_{X}) )$. Moreover, $\Upsilon ( \overline{\{ \delta_z : z \in \Omega \} }^{\,\sigma}) =   \overline{\{ \delta_z : z \in \Omega \} }^{\,\sigma}$ where the former $\sigma$ denotes the Gelfand topology on $\M (\A_{w^* u} (B_{X^{**}}))$ and the latter one is the Gelfand topology on $\M (\A_{w u} (B_{X}))$. Thus, the assertion follows from Theorem \ref{prop:separable_dual}. 
\end{proof} 

Recall from Lemma \ref{lem:weak_unif_conti} and Proposition \ref{prop:AP_wu_spectrum} that when $X^*$ has the approximation property, the spectrum $\M (\A_{wu} (B_X))$ coincides with $\{ \delta_z : z \in \overline{B}_{X^{**}} \}$. Using this fact, we observe that the Gelfand topology $\sigma$-closure of a set of point evaluations in $\M (\A_{wu} (B_X))$ can be expressed in a simpler way as follows. 

\begin{lemma}\label{lem:gamma_closure}
Let $X$ be a Banach space whose dual has the approximation property and $\Gamma$ be a subset of $B_X$. Then $\overline{ \{\delta_z : z \in \Gamma \}}^{\,\sigma} = \left\{ \delta_z : z \in \overline{\Gamma}^{w^*} \right\}$ in $\M (\A_{wu} (B_X))$. 
\end{lemma} 

\begin{proof}
Let $\phi$ be an element of $\overline{ \{\delta_z : z \in \Gamma \}}^{\,\sigma}$, then there exists a net $(z_\alpha)$ in $\Gamma$ such that $(\delta_{z_\alpha})$ converges to $\phi$. By Proposition \ref{prop:AP_wu_spectrum}, $\phi = \delta_{z_0}$ for some $z_0 \in \overline{B}_{X^{**}}$. Then $(z_\alpha) = ( \pi (\delta_z (\alpha) ) )$ converges weak-star to $\pi (\phi) = z_0$. It follows that $\phi \in \left\{ \delta_z : z \in \overline{\Gamma}^{w^*} \right\}$. 

Conversely, let $z_0 \in \overline{\Gamma}^{w^*}$ be given. Then there is a net $(z_\alpha)$ in $\Gamma$ so that $(z_\alpha)$ converges weak-star to $z_0$. For any $f \in \A_{wu} (B_X)$, the net $( \widetilde{f} (z_\alpha))$ converges to $ \widetilde{f} (z_0)$ since $ \widetilde{f} \in \A_{w^* u } (B_{X^{**}})$. This implies that $(\delta_{z_\alpha})$ converges to $\delta_{z_0}$ in $\M (\A_{wu} (B_X))$; hence $\delta_{z_0} \in \overline{ \{ \delta_z : z \in \Gamma \}}^{\, \sigma}$. 
\end{proof}

%By using Proposition \ref{} combined with Remark \ref{}, or using directly the above Proposition \ref{prop:separable_dual}, we have the following result. 

\begin{cor}\label{cor:sep_ap}
Let $X$ be a Banach space whose dual space is a separable Banach space with the approximation property. Then the Shilov boundary for $\widehat{\A_{w u} (B_{X})}$ coincides with $\left \{ \delta_z : z \in \overline{\Ext_{\mathbb{C}} (\overline{B}_{X^{**}}) }^{w^*}\right\}$.
\end{cor}

One of direct consequences of Corollary \ref{cor:sep_ap} is that if $X$ is a strictly convex Banach space whose dual space is separable and has the approximation property, such as $\ell_p$ ($1<p<\infty$), then $\partial \widehat{\A_{w u} (B_{X})} = \left\{ \delta_z : z \in \overline{S_X}^{w^*} \right\} = \left\{ \delta_z : z \in \overline{B}_{X^{**}} \right\}$.

Moreover, notice that if $X$ is $c_0$, the Tsirelson or the Tsirelson-James space, then $ \Pol_{wu} (^N X) = \Pol (^N X)$; hence $\A_u (B_X) = \A_{wu} (B_X)$ (see \cite{AD} and \cite{Gu}). Besides, it is known that $\A_u (B_{d_{*} (w,1)}) = \A_{wu} (B_{d_{*} (w,1)})$ whenever $w \notin \ell_p$ for every $p \in \N$ \cite[Proposition 2.6]{AAM}. As all the just mentioned Banach spaces have separable dual spaces satisfying the approximation property, Corollary \ref{cor:sep_ap} yields the following. 
%Thus, we have the following corollary thanks to Lemma \ref{}. 

\begin{cor}\label{cor:c_0,Tsi,Lor}
If a Banach space $X$ is one of the following:
\begin{enumerate}
\setlength\itemsep{0.4em}
\item[(a)] $c_0$
\item[(b)] Tsirelson's space,
\item[(c)] Tsirelson-James space, 
\item[(d)] $d_{*} (w,1)$ with $w \notin \ell_p$ for every $p \in \N$, 
%\item[(d)] A locally $c$-uniformly convex order continuous sequence space,
\end{enumerate} 
then the Shilov boundary for $\widehat{\A_u (B_X)}=\widehat{\A_{wu} (B_X)}$ coincides with $\left \{ \delta_z : z \in \overline{\Ext_{\mathbb{C}} (\overline{B}_{X^{**}}) }^{w^*}\right\}$. 
\end{cor} 

Let us remark that, in contrast to (d) of Corollary \ref{cor:c_0,Tsi,Lor}, it is known that there is no minimal closed boundary for $\A_{wu} (B_{d_{*} (w, 1)})$ \cite[Corollary 3.6]{AAM}. Moreover, the item (a) of Corollary \ref{cor:c_0,Tsi,Lor} reproves Proposition \ref{prop:A_u_c_0} since $\Ext_{\mathbb{C}} (\overline{B}_{\ell_\infty} ) = \mathbb{T}^{\mathbb{N}} $.

Finally, we would like to present the descriptions of the Shilov boundaries for $\M(A_{wu} (B_X))$ for Banach spaces $X$ which have the Radon-Nikod\'ym property and whose duals have the  approximation property. 
To this end, we observe first the following lemma by using the result of \cite{CLS} concerning on the denseness of strong peak functions. 

\begin{lemma}\label{lem:reflexive_AP_strong_peak}
Suppose that a Banach space $X$ has the Radon-Nikod\'ym property and $X^*$ has the approximation property. Let $x_0 \in S_X$ be a strong peak point for $\A_{wu} (B_X)$ and $\eps >0$ be given. Then there is a strong peak point $x_1$ for $\A_a (B_X)$ such that $\|x _0 - x_1 \| < \eps$. 
\end{lemma}  

\begin{proof}
Let $f_0 \in \A_{wu} (B_X)$ be the strong peak function corresponding to $x_0$. Take $\delta >0$ such that $|f(x)| < 1- \delta$ whenever $\|x - x_0\| > \eps$. 
For $0 < r < 1$, consider $f_r (x) := f(rx)$ and note that $(f_r)$ converges to $f$ uniformly on $B_X$ as $r$ increases to $1$. On the other hand, by Lemma \ref{lem:weak_unif_conti}, the Taylor coefficients of $f_r \vert_{B_X}$ at $0$ are also weakly uniformly continuous on bounded sets, and the Taylor series expansion of $f_r \vert_{B_X}$ converges uniformly on $B_X$. 
It follows that there exists a continuous polynomial $Q$ which is weakly uniformly continuous on bounded sets such that $\| Q - f_0 \| < \frac{\delta}{3}$. 
Since $X^*$ has the approximation property, notice from \cite[Proposition 2.8]{Dineen} that there exists a finite type polynomial $R$ such that $\| R - Q \| < \frac{\delta}{3}$. Notice that $X$ has the Radon-Nikod\'ym property, so the proof of  \cite[Theorem 4.4]{CLS} implies that there exists a finite type polynomial $R'$ such that $\|R' \| < \frac{\delta}{3}$ and $P:= R + R' \in \A_a (B_X)$ peaks strongly at $x_1 \in S_X$. It follows that $\| f_0 - P \| < \delta$. In particular, we have that $|f_0 (x_1)| \geq 1 - \delta$. This implies that $\|x_1 - x_0\| \leq \eps$. Therefore, $x_1$ is the desired strongly peak point for $\A_a (B_X)$. 
\end{proof}

\begin{theorem}\label{prop:reflexive_AP_shilov}
Suppose that a Banach space $X$ has the Radon-Nikod\'ym property and $X^*$ has the approximation property. Then the Shilov boundary for $\widehat{\A_{wu} (B_X)}$ coincides with $\left\{ \delta_z : z \in \overline{\partial \A_{wu} (B_X) }^{w^*} \right\}$%^{\,\sigma}$. %\subset \partial \widehat{\A_u} \subset \overline{\{ \delta_x :x \in B_X\}}^{\,\sigma}.
\end{theorem} 

\begin{proof} We divide the proof in two steps.  

\noindent \emph{STEP I}. If $x_0 \in S_X$ is a strong peak point for $\A_{wu} (B_X)$, then $\delta_{x_0}$ belongs to $\partial \widehat{\A_{wu}}$.

If not, there exists an open neighborhood $U$ of $\delta_{x_0}$ such that $U \cap \partial \widehat{\A_{wu}} = \emptyset$. Without loss of generality, we write 
\[
U = \bigcap_{j=1}^{N} \{ \phi \in \M(\A_{wu} (B_X)) : |(\phi - \delta_{x_0})(f_j) | < \eps \}
\]
for some $f_1, \ldots, f_N \in \A_{wu} (B_X)$ and $\eps >0$. Take $r >0$ so small so that whenever $x \in B_X$ satisfies $\| x - x_0\| < r$, then $|f_j (x) - f_j (x_0)| < \eps$ for any $j=1,\ldots, N$. By Lemma \ref{lem:reflexive_AP_strong_peak}, there exists a strong peak point $x_1 \in S_X $ for $\A_a (B_X)$ such that $\|x_1 - x_0\| < r$. 
By observing that Theorem \ref{prop:strong_peak} can also be applied to the algebra $\A_{wu} (B_X)$, we can deduce that $\delta_{x_1}$ is a peak point for $\widehat{\A_{wu} (B_X)}$. Notice that $\delta_{x_1}$ belongs to the open set $U$; hence $\delta_{x_1} \in U \cap \partial \widehat{\A_{wu}}$; which is a contradiction. Thus, 
\[
\{ \delta_x : x \text{ is a strong peak point for } \A_{wu} (B_X)\}  \subseteq \partial \widehat{\A_{wu}}.
\]

\noindent \emph{STEP II}. $\{ \delta_z : z \in \partial \A_{wu} (B_X) \} \subseteq \partial \widehat{\A_{wu}}$. 

Arguing as in \cite[Theorem 4.4]{CLS}, we can see that the set of strong peak points for $\A_{wu} (B_X)$ is dense in $\partial \A_{wu} (B_X)$. Thus, for $z \in \partial \A_{wu} (B_X)$, there is a sequence $(x_n) \subset S_X$ of strong peak points for $\A_{wu} (B_X)$ such that $\| x_n - z\| \rightarrow 0$. As we have checked in STEP I, each $\delta_{x_n}$ is an element of $\partial \widehat{\A_{wu}}$. Thus, $\delta_z$ belongs to $\partial \widehat{\A_{wu}}$ since $(\delta_{x_n})$ converges in the Gelfand topology to $\delta_z$ in $\M(\A_{wu} (B_X))$ and $\partial \widehat{\A_{wu}}$ is closed. Therefore, the assertion follows.

Finally, arguing as in Proposition \ref{prop:boundary_norming}, we see that $\partial \widehat{\A_{wu}} \subseteq \overline{\{ \delta_x : x \in \partial \A_{wu} (B_X) \}}^{\,\sigma} = \left \{ \delta_x : x \in \overline{ \partial \A_{wu} (B_X) }^{w^*} \right \}$, where the second equality holds by Lemma \ref{lem:gamma_closure}. 
\end{proof}

\begin{cor}      
If a Banach space $X$ is one of the following:
\begin{enumerate}
\setlength\itemsep{0.4em}
\item[(a)] A uniformly convex space having the approximation property, 
\item[(b)] $\ell_1$,
\item[(c)] $d(w,1)$ with a decreasing sequence $w=(w_n)$ of positive real numbers satisfying $w_1 = 1$ and $w \in c_0 \setminus \ell_1$,
%\item[(d)] A locally $c$-uniformly convex order continuous sequence space,
\end{enumerate} 
then $\partial \widehat{\A_{wu} (B_X)} = \M (\A_{wu} (B_X)) = \{ \delta_z : z \in \overline{B}_{X^{**}} \}$. 
\end{cor} 

\begin{proof}
We already observed in Section \ref{section_1} that if $X$ is one of the Banach spaces (a), (b) or (c), then the unit sphere $S_X$ is the Shilov boundary for $\A = \A_u (B_X)$ or $\A_\infty (B_X)$, thus it is also the Shilov boundary for $\A_{wu} (B_X)$. Moreover, $X$ has the Radon-Nikod\'ym property and $X^*$ has the approximation property. By Theorem \ref{prop:reflexive_AP_shilov}, it follows that the Shilov boundary $\partial \widehat{\A_{wu} (B_X)}$ coincides with $\left\{ \delta_z : z \in \overline{S_X}^{w^*} \right\}$ which is nothing but $\{ \delta_z : z \in \overline{B}_{X^{**}} \} = \M (\A_{wu} (B_X))$. 
\end{proof} 

%Let $\A_{w^* u } (B_{X^{**}})$ be the Banach algebra of all complex-valued holomorphic functions on $B_{X^{**}}$ which are weak star uniformly continuous on $B_{X^{**}}$. 

%\proof[Acknowledgements]

%The authors would like to thank Manuel Maestre for fruitful conversations on the topic of the paper.

% and Abraham Rueda Zoca .% The author also wishes to thank  for bringing the paper \cite{DFS} to our attention and for communicating some results on uniformly strongly exposed points.
%afdsa

\end{document}